\documentclass[11pt,notitlepage,a4paper,reqno,ps]{amsart}
\usepackage{eurosym}
\usepackage{amssymb}
\usepackage{amstext}
\usepackage{color}
\usepackage{xcolor}
\usepackage{amsmath,amsthm,amssymb,amscd,epsfig,esint, mathrsfs,graphics,color}
 \usepackage{wrapfig}
\usepackage{verbatim}
\usepackage[latin1]{inputenc}
\xdefinecolor{azzurro}{rgb}{0.2,0.9,0.9}
\xdefinecolor{violetto}{rgb}{0.5,0.5,0.9}
\xdefinecolor{marrone}{rgb}{0.8,0.4,0.2}

\newcommand{\loc}{\textnormal{loc}}
\newcommand{\medint}{-\kern  -,375cm\int}

\definecolor{ora}{rgb}{0.8,0.2,0.1}
\definecolor{vio}{rgb}{0.5,0,0.5}
\definecolor{gre}{rgb}{0.1,0.6,0}
\definecolor{verde}{rgb}{0,0.7,0.4}
\newenvironment{michelarev}{\color{azzurro}}{\color{black}}
\newcommand{\bmicr}{\begin{michelarev}}
\newcommand{\emicr}{\end{michelarev}}
\theoremstyle{plain}
\newtheorem{theorem}{Theorem}[section]

\newtheorem{lemma}[theorem]{Lemma}

\theoremstyle{definition}

\theoremstyle{remark}
\newtheorem{remark}[theorem]{Remark}
\theoremstyle{plain}

\definecolor{dg}{rgb}{0.01, 0.75, 0.24}

\def\R{\mathbb{R}}
\numberwithin{equation}{section} \makeatletter

\renewcommand{\p@enumi}{\thesection.}
\makeatother  \allowdisplaybreaks
\email{michela.eleuteri@unimore.it}
  \email{antonia.passarellidinapoli@unina.it}

\begin{document}
\title[On the validity of variational inequalities for obstacle problems]{On the validity of variational inequalities for obstacle problems with non-standard growth}
\author[M. Eleuteri -- A. Passarelli di Napoli]{Michela Eleuteri -- Antonia Passarelli di Napoli}
\address{Dipartimento di Scienze Fisiche, Informatiche e Matematiche, Universit\`a degli Studi di Modena E Reggio Emilia,  
via Campi 213/b, 41125 Modena, Italy}
\address{Dipartimento di Matematica e Applicazioni ``R. Caccioppoli''
\\
Universit\`a degli Studi di Napoli ``Federico II''\\
Via Cintia, 80126, Napoli (Italy)}
\thanks{\textit{Acknowledgements.}
The work of the authors is supported by GNAMPA (Gruppo Nazionale per l'Analisi Matematica, la Probabilit\`a e le loro Applicazioni) of INdAM (Istituto Nazionale di Alta Matematica), by MIUR through the project FFABR and by the University of Modena and Reggio Emilia through the project FAR2019 ``Equazioni differenziali: problemi evolutivi, variazionali ed applicazioni'' (coord. Prof. M. Manfredini). This research was performed while A. Passarelli di Napoli was visiting the University of Modena and Reggio Emilia and M. Eleuteri was visiting the University of Naples ``Federico II". The hospitality of both Institutions is warmly adknowledged}.

\begin{abstract}
{The aim of the paper is to show that the solutions to variational problems with non-standard growth conditions satisfy a corresponding variational inequality without any smallness assumptions on the gap between growth and coercitivity exponents. Our results rely on techniques based on Convex Analysis that consist in establishing duality formulas and pointwise relations between minimizers and corresponding dual maximizers, for suitable {approximating problems}, that are preserved passing to the limit. In this respect we are able to show that the right class of competitors are the functions with finite energy, in agreement with the unconstrained results.}
\end{abstract}

\maketitle

\begin{center}
\fbox{\today}
\end{center}

\section{Introduction}

More than 30 years ago, the celebrated papers by Marcellini \cite{M89}, \cite{M91} opened the way to the study of the regularity properties of minimizers {of integral functionals} with non-standard growth conditions. Since then, many contributions appeared in several directions and many problems have been solved; however not all the questions have been addressed in an exaustive way, in particular for what  concerns  the obstacle problems. 
\\
It is well known that, for both constrained and unconstrained problems, the regularity of the solutions often comes from the fact that are also extremals, i.e. they solve a corresponding variational inequality or equality.
\\
Actually, in the recent paper \cite{CEP} the authors, dealing with the question of Lipschitz continuity for minimizers of the obstacle problem, were forced to deal with the relation between minima and extremals, in the sense of solutions to a corresponding variational inequality. In that specific situation, this problem has been solved due to a suitable higher differentiability result and  imposing a smallness condition on the gap between the coercivity and the growth exponent of the lagrangian. 
\\
Already for unconstrained minimizers with non-standard growth, the relation between extremals and minima is an issue that required a careful investigation.
Indeed, a direct derivation of such a relation can be obtained in a trivial way only if the gap between the growth and the ellipticity exponent satisfies  a suitable smallness condition.
\\
Otherwise, using a regularization procedure and convex duality theory,  much stronger results have been obtained  by Carozza, Kristensen and Passarelli di Napoli for unconstrained minimizers (see  \cite{CKPPisa}, \cite{CKP Comm}). 
\\
As far as we know, such investigation has not been carried out for constrained minimizers. 

The aim of this paper is to fill this gap, by finding conditions so that the solutions to variational obstacle problems with non standard growth conditions satisfy a corresponding variational inequality.

More precisely, let us consider a class to variational obstacle problems of the form
 \begin{equation}
\label{obst-def0}
\min\left\{\int_\Omega F(Dz): z\in \mathbb{K}_{\psi}^F(\Omega)\right\},
\end{equation}
where $\Omega$ is a bounded open set of $\mathbb{R}^n$, $n \ge 2$.
The function $\psi:\,\Omega \rightarrow [- \infty, + \infty)$, called \textit{obstacle}, is such that 
\begin{equation}\label{assumpsi}
F(D \psi) \in L^1(\Omega)	
\end{equation} and 
the class $\mathbb{K}_{\psi}^F(\Omega)$ is defined as 
\begin{equation}
\label{classeF}
\mathbb{K}_{\psi}^{F}(\Omega) := \left \{z \in u_0+{W^{1,p}_0(\Omega)}: z \ge \psi  \,\, \textnormal{a.e. in $\Omega$},\,\,\, F( Dz) \in L^1(\Omega) \right\},
\end{equation}
where $u_0$ is a fixed boundary value such 
that 
\begin{equation}\label{assumuzero}
F(D u_0) \in L^1(\Omega).	
\end{equation}
To avoid trivialities, in what follows we shall assume that $\mathbb{K}_{\psi}^F(\Omega)$ is not empty. 
We shall consider integrands $F: \mathbb{R}^n \rightarrow \mathbb{R}$ of class $\mathcal{C}^1$ and satisfying  the following growth and strict convexity assumptions:
$$\ell|\xi|^p \le F(\xi) \le \, L \, (1 + |\xi|^q) \eqno{{\rm (H1)}}$$
$$ \nu \, |V_p(\xi) - V_p(\eta)|^2 \le \, F(\xi) - F(\eta) - \langle F'(\eta), \xi - \eta \rangle \eqno{{\rm (H2)}}$$
for all $\xi, \eta \in \mathbb{R}^n$, for $0 < \ell < L$, $\nu>0$  and $1 < p \le q < \infty$ and where we used the customary notation
\begin{equation}
\label{Vp}
V_p(\xi)=(1+|\xi|^2)^{\frac{p-2}{4}}\xi.
\end{equation}
To simplify the statement of our main result, we shall assume that the integrand $F$ satisfies a sort of $\Delta_2$ condition, i.e.
$$F(\lambda \, \xi) \le \, C(\lambda) \, F(\xi)  \eqno{{\rm (H3)}},$$
for every real positive $\lambda >1$ and every $\xi \in \mathbb{R}^n$.
\\
Actually, without (H3), our result holds true supposing, beside \eqref{assumuzero} that $F(cDu_0)\in L^1(\Omega)$, for some constant $c>1$.
\medskip
\begin{remark}
\label{rem-serena}
Let us notice that, by replacing $u_0$ by $\tilde{u}_0 = \max \{u_0, \psi\}$, we
may assume that the boundary value function $u_0$ satisfies $u_0 \ge \,
\psi$ in $\Omega$. Indeed $\tilde{u}_0  = (\psi - u_0)^+ + u_0$ and since
\[
0 \le (\psi - u_0)^+ \le (u - u_0)^+ \in W^{1,p}_0(\Omega),
\] 
the function $(\psi - u_0)^+$, and hence $u - \tilde{u}_0,$ belongs to
$W^{1,p}_0(\Omega)$. 
Moreover assumptions \eqref{assumpsi} and \eqref{assumuzero} imply $F(D\tilde u_0)\in L^1(\Omega)$. Indeed we have
\begin{eqnarray*}
\int_\Omega F(D\tilde u_0)\,dx&=&\int_{\Omega\cap\{ u_0\ge \psi\}} F(D u_0)\,dx+\int_{\Omega\cap\{ u_0< \psi\}} F(D \psi)\,dx\cr\cr
&\le& \int_\Omega \Big(F(D u_0)+F(D \psi)\Big)\,dx<+\infty,	
\end{eqnarray*}
where we used that $F(\xi)\ge 0$, by virtue of the left inequality in (H1).
\end{remark}
In view of the previous remark, from now on, without loss of generality, we shall suppose $$u_0 \ge \,
\psi\qquad\qquad\text{a.e.\,\,in\,\,}\Omega.$$
It is worth mentioning that if $G$ is a $\mathcal{C}^1$ function satisfying (H1) and (H2) with $p = q$, i.e. $G$ satisfies standard $p-$growth conditions, the minimization problem reduces to 
\begin{equation}
\label{p-obst-def0}
\min\left\{\int_\Omega G(Dz): z\in \mathcal{K}_{\psi}(\Omega)\right\},
\end{equation}
where
\begin{equation}
\label{classeA}
\mathcal{K}_{\psi}(\Omega) := \left \{z \in u_0+{W^{1,p}_0(\Omega)}: z \ge \psi \,\, \textnormal{a.e. in $\Omega$} \right\}
\end{equation}
and the assumptions $F(D\psi), F(Du_0) \in L^1(\Omega)$ reduce in turn to $\psi, u_0 \in W^{1,p}(\Omega)$.
\\
In this case, because of the standard growth conditions, it is well known that, if  $u\in u_0+W^{1,p}_0(\Omega)$ is a solution to \eqref{p-obst-def0}, then the corresponding variational inequality 
\begin{equation}\label{var-ineq-stand}
	\int_\Omega \langle G'(Du), Dz-Du\rangle\,dx\ge 0
\end{equation}
holds true, for every $z\in \mathcal{K}_\psi(\Omega)$. This can be proved by observing that, since $\mathcal{K}_{\psi}(\Omega)$ is a convex set, the function $u + \varepsilon (z - u) = (1 - \varepsilon) u + \varepsilon z \in \mathcal{K}_{\psi}(\Omega)$ is an admissible variation for each $0 \le \varepsilon < 1.$
\\
On the other hand, if $u \in \mathcal{K}_{\psi}(\Omega)$ and $\varphi \ge 0,$ with $\varphi \in \mathcal{C}^{\infty}_0(\Omega),$ then $u + \varphi \in \mathcal{K}_{\psi}(\Omega)$ and thus, if $u$ is a solution to \eqref{p-obst-def0}, then also the following inequality holds
\begin{equation}\label{var-ineq-varphi}
	\int_\Omega \langle G'(Du), D\varphi\rangle\,dx\ge 0
\end{equation}
for all $\varphi \in \mathcal{C}^{\infty}_0(\Omega)$, $\varphi \ge \, 0$. 
%

Our goal is to show that, if we assume {\eqref{assumpsi} and} \eqref{assumuzero}, {the} solutions to obstacle problems with non standard growth conditions solve the corresponding variational inequalities, without any restriction on the gap $\frac{q}{p}$. Moreover we will show that the right class of competitors are the functions with  finite energy  and that, in case of standard growth conditions, this coincides with  $\mathcal{K}_{\psi}(\Omega)$. 
\begin{theorem}
\label{teo2}
Let $F: \mathbb{R}^n \rightarrow \mathbb{R}$ be a $\mathcal{C}^1$ function satisfying {\rm (H1)}, {\rm (H2)} and {\rm (H3)}.  Assume moreover that \eqref{assumpsi} and \eqref{assumuzero} hold true. If $u\in \mathbb{K}_\psi^F(\Omega)$ is  the solution to the obstacle problem \eqref{obst-def0}, then 
\begin{equation}
\label{tuttoL1}
F^*(F'(Du)) \in L^1(\Omega) \qquad \qquad \langle F'(Du), Du \rangle \in L^1(\Omega)
\end{equation}
and 
\begin{equation}
\label{div}
\mathrm{div}F'(Du)\le 0 
\end{equation}
in the distributional sense.
Moreover the following variational inequality 
\begin{equation}\label{pmlim}
\int_{\Omega} \langle F'(Du), Dz-Du \rangle \ge \, 0
\end{equation}
holds for all $z \in {\mathbb{K}_\psi^F(\Omega)}$ such that 
\begin{equation}\label{pm}
F(\pm Dz)\in L^1(\Omega)	.
\end{equation}
\end{theorem}

Here $F^*$ denotes the polar, or Fenchel conjugate, of the convex continuous function $F$, introduced in \eqref{Fconiug} of Subsection 2.4.
\\
Hence, in view of Theorem \ref{teo2}, $u$ in particular solves the corresponding variational inequality and $F'(Du) \in L^{q'}(\Omega; \mathbb{R}^n)$ with $q' = \frac{q}{q-1}$.

It is worth noticing that, if there exists $f: [0,+\infty)\to [0,+\infty)$ such that  $$F(\xi)=f(|\xi|), $$ then  assumption \eqref{pm} is trivially satisfied. On the other hand, in order to have \eqref{pm} satisfied for every $z\in \mathbb{K}_\psi^F(\Omega)$, it suffices to assume that
\begin{equation}
	F(\xi)=F(-\xi).
\end{equation}
 As remarked in \cite{cianchi}, under this assumption  $F(\xi)$ needs not to depend on the length of $\xi$ nor to be the sum of its components $\xi_i$. Indeed, an example of $F(\xi)$ satisfying our assumptions is
$$F(\xi)=|\xi_1-\xi_2|^q+|\xi_1+\xi_2|^p\log^\alpha(1+|\xi_1|)\qquad\quad \xi\in \mathbb{R}^2,$$
with $2\le p\le q$.

\medskip

In case the gap $\frac{q}{p}$ satisfies a suitable smallness assumption and if the obstacle $\psi\in W^{1,q}_{\loc}(\Omega)$, we are able to prove that the solution to problem
\begin{equation}
\label{p-obst-defF}
\min\left\{\int_\Omega F(Dz): z\in \mathcal{K}_{\psi}(\Omega)\right\},
\end{equation}
{with $\mathcal{K}_{\psi}(\Omega)$ as in \eqref{classeA},} solves the corresponding variational inequality without any regularity on the boundary datum $u_0$. Moreover, we can prove that the solution to \eqref{p-obst-defF} locally belongs to $W^{1,q}_{\loc}(\Omega)$ if the obstacle $\psi$ locally belongs to $W^{1,q}$. 
 More precisely, we have the following 
\begin{theorem}
\label{teo3}
Let $F: \mathbb{R}^n \rightarrow \mathbb{R}$ be a $\mathcal{C}^1$ function satisfying {\rm (H1)}, {\rm (H2)} and {\rm (H3)}. Assume  that
$$D\psi\in W^{1,q}_{\loc}(\Omega) $$  and let 
 $u \in \mathcal{K}_{\psi}(\Omega)$ be the solution to the obstacle problem \eqref{p-obst-defF}. If
\begin{equation}
\label{gap-pq}
1 < p \le q < \frac{np}{n-1}
\end{equation}
then {$u$ is such that }
\[
F^*(F'(Du)) \in L^1_{\loc}(\Omega) \qquad \qquad \langle F'(Du), Du \rangle \in L^1_{\loc}(\Omega)
\]
and
$${\rm div} F'(Du) \le \, 0$$
locally, in the distributional sense and moreover $u \in W^{1, q}_{\loc}(\Omega)$.
\end{theorem}

\begin{remark}
Note that, arguing as in \cite{CKPPisa}, in case
\[
\frac{np}{n-1} \le q < p^*
\]
and $D \psi \in W^{1,q}_{\loc}(\Omega)$, then the solution $u$ to the obstacle problem \eqref{p-obst-def0} belongs to $W^{1,r}_{\loc}(\Omega)$ for all $r < \bar{p}$ being
\[
\bar{p} := \frac{np}{n - \frac{p}{p-1} \left (1 - n \left (\frac{1}{p} - \frac{1}{q} \right ) \right )}.
\]
\end{remark}


This result is particularly important in order to prevent the Lavrentiev phenomenon tha may occurr in the case of anisotropic growth conditions.
\\
\\
Let's mention a few words about the techniques employed. Our Lagrangian $F$ has been suitably approximated  by strictly convex and uniformly
elliptic integrands $F_k,$  in order to facilitate a systematic use of the dual problems, in the sense
of Convex Analysis. The minimizers of $F_k$, say $u_k$, strongly converge in $W^{1,p}$ to the minimizer $u$ of \eqref{obst-def0}  and to every such minimizer $u_k$ we
can associate the solutions of certain dual maximization problems for divergence-measure fields. 
\\
Next, we establish duality formulas and pointwise relations between minimizers and dual maximizers  that are preserved in passing to the limit. Such
estimates will provide conditions in order for the variational inequality to hold for
a constrained minimizer. The statement and the proofs of our results, that are the counterpart of those in \cite{CKPPisa} concerning the unconstrained setting, rely on a suitable version of Anzellotti type pairing which involve general divergence-measure fields and specific representation of Sobolev functions, and which reduces to integration by part formula once the correct summability is required on the fields involved. 

\section{Notations and Preliminary Results}\label{prelim}

\bigskip

\noindent
In this paper we shall denote by $C$ or $c$  a
general positive constant that may vary on different occasions, even within the
same line of estimates.
Relevant dependencies  will be suitably emphasized using
parentheses or subscripts.  In what follows, $B(x,r)=B_r(x)=\{y\in \R^n:\,\, |y-x|<r\}$ will denote the ball centered at $x$ of radius $r$.
We shall omit the dependence on the center and on the radius when no confusion arises.

For the auxiliary function $V_{p}$, introduced in \eqref{Vp}, we recall the following
estimate (see the proof of \cite[Lemma 8.3]{Giusti}):

\begin{lemma} \label{Vi}
Let $1<p<\infty$. There exists a constant $c=c(n,p)>0$
such that
$$
c^{-1}\Bigl( 1+| \xi |^2+| \eta |^2 \Bigr)^{\frac{p-2}{2}}\leq
\frac{|V_{p}(\xi )-V_{p}(\eta )|^2}{|\xi -\eta |^2} \leq
c\Bigl( 1+|\xi |^2+|\eta |^2 \Bigr)^{\frac{p-2}{2}}
$$
for any $\xi$, $\eta \in \R^n$.
\end{lemma}


\subsection{Besov Spaces}
 Let us recall that, for every   function $f:\mathbb{R}^{n}\to \mathbb{R}$ 
the finite difference operator is defined by
$$
\tau_{s,h}f(x)=f(x+he_{s})-f(x)
$$
where $h\in\mathbb{R}$, $e_{s}$ is the unit vector in the $x_{s}$
direction and $s\in\{1,\ldots,n\}.$
\\
Let $1\le p< \infty$ and
$0<\alpha<1$. The Besov space $B^{\alpha}_{p,\infty}(\mathbb{R}^n)$ consists of the function   $v\in L^p(\mathbb{R}^n)$ such that
   $$
[v]_{\dot{B}^\alpha_{p,\infty}(\R^n)} =\sup_{h\in\R^n}   \left(\int_{\R^n}\frac{|v(x+h)-v(x)|^p}{|h|^{\alpha p }} dx\right)^\frac{1}{p} <\infty.$$
One can simply take supremum over $|h|\leq \delta$ and obtain an equivalent norm.  By construction, $B^\alpha_{p,\infty}(\R^n)\subset L^p(\R^n)$. 
\\
Given a domain $\Omega\subset\R^n$ , we say that $v$ belongs to the local Besov space $B^\alpha_{p,\infty,\loc}$ if $\varphi\,v$ belongs to the global Besov space $B^\alpha_{p,q}(\R^n)$ whenever $\varphi$  belongs to the class $\mathcal{C}^\infty_0(\Omega)$ of smooth functions with compact support contained in $\Omega$. One also has the following fractional version of Sobolev embeddings (see \cite{Ha})\noindent

\begin{lemma}\label{lefrac} Let  $f\in
L^{p}(B_{R})$. Suppose that there exist $\rho\in(0,R)$,   $0<\alpha<1$ and  $M>0$ such that
$$
\sum_{s=1}^{n}\int_{B_{\rho}}|\tau_{s,h}f(x)|^{p}\,dx\leq
M^{p} |h|^{p\alpha },
$$
for every $h$ such that $|h|<\frac{R-\rho}{2}$. Then $f\in L^{\frac{pn}{n-p\beta}}(B_{\rho})$ for every $\beta\in (0,\alpha)$
and
$$
||f||_{L^{\frac{pn}{n-p\beta}}(B_{\rho})}\leq
c\left(M+||f||_{L^{p}(B_{R})}\right),
$$
with $c= c(n, N,R,\rho,\alpha,\beta)$.
\end{lemma} 
We also have the following embedding theorem that relates Sobolev and Besov spaces that can be deduced with the arguments of \cite[Section 30-32]{tartar}. We give the proof for the sake of clarity.
\begin{theorem}\label{thtartar}
Let $\Omega\subset\mathbb{R}^n$. The continuous embedding
$$W^{1,p}_{\mathrm{loc}}(\Omega)\hookrightarrow B^\alpha_{q,\infty,\mathrm{loc}}(\Omega)$$
holds, provided $1<p<q<+\infty$ and $\alpha=1-n\left(\frac{1}{p}-\frac{1}{q}\right).$
\end{theorem}
\begin{proof}
Let $v\in W^{1,p}_{\mathrm{loc}}(\Omega)$ and fix balls $B_r\subset B_R\Subset\Omega$. We start with the following interpolation inequality stated at formulas (30.8) and (30.9) in \cite{tartar},  that reads as follows
\begin{equation}\label{interp1}
||v||_{L^q(B_r)}\le c||v||_{L^p(B_r)}^{1-\eta}||Dv||_{L^q(B_r)}^\eta 	
\end{equation}
where  
$$\eta=n\left(\frac{1}{p}-\frac{1}{q}\right).$$
For $u\in W^{1,p}_{\mathrm{loc}}(\Omega)$, we use \eqref{interp1} with $\tau_h u$ in place of $v$, we get
\begin{equation}\label{interp2}
||\tau_h u||_{L^q(B_r)}\le c||\tau_h u||_{L^p(B_r)}^{1-\eta}||D\tau_h u||_{L^p(B_r)}^\eta \le c|h|^{1-\eta}||Du||_{L^p(B_R)},	
\end{equation}
where we used that, for $u\in W^{1,p}_{\mathrm{loc}}(\Omega),$ it holds 
$$ \int_{B_r}|\tau_h u|^p\,dx\le c(n,p) |h|^p \int_{B_R}|D u|^p\,dx$$
and that
$$ \int_{B_r}|Du(x+h)|^p\,dx\le c(n,p)  \int_{B_R}|D u|^p\,dx.$$
The conclusion follows by \eqref{interp2}, recalling the value of $\eta$.
\end{proof}
\medskip

\subsection{Some approximation results}
\medskip
Now, we state a useful approximation lemma whose proof can be found in \cite[Proposition 3.1]{CKPPisa} and  that will be needed in the sequel. 

\begin{lemma}
\label{lemmaCGM}
Let $F:\R^{n}\to {\mathbb{R}}$ be a $\mathcal{C}^1$ function 
satisfying  assumptions \textnormal{(H1)--(H2)}. 
 Then there exists a sequence
$({F_k})$ of $\mathcal{C}^1$ functions
${F_k}:\R^{n}\to {\mathbb{R}}$, monotonically convergent to $F$,
such that
\begin{itemize}
      \item[(I)] for every  $\xi
      \in
  \R^{n}$, and  for every $k_1 < k_2$, it holds 
 $${F_{k_1}}(\xi)\le {F_{k_2}}(\xi)\le F(\xi),$$
  \item[(II)]  for every  $\xi\in
   {\R^{n}}$,   we have 
	$$c(p, \nu)\, |V_p(\xi) - V_p(\eta)|^2 \le \, F_k(\xi) - F_k(\eta) - \langle F_k'(\eta), \xi - \eta \rangle $$
   \item[{(III)}] for every  $\xi
      \in
  \R^{n}$,
 there exist constants $L_{0},L_{1}$, independent of $k$, and
 ${\overline L}_{1}$, depending on $k$,
 such that
  \begin{equation*}
 \begin{split}
 L_{0}(|\xi|^{p}-1)&\le {F_k}(\xi)\le L_{1}(1+|\xi|)^{q},
 \\
{F_k}(\xi)&\le {\overline L}_{1}(k)(1+|\xi|)^{p}.
\end{split}
\end{equation*}
\item[(IV)] If $\xi_k\to \xi$ then $F'(\xi_k)\to F'(\xi)$ locally uniformly.
\end{itemize}
\end{lemma}

Actually, a careful inspection of the proof of \cite[Proposition 3.1]{CKPPisa} reveals that there exists a sequence $\mu_k\in \mathbb{R}$ such that 
$$\lim_{k\to \infty}\mu_k=0$$
and 
$$ F_k(\xi)\ge L_0|\xi|^p-\mu_k\qquad \text{for every}\,\, k\in \mathbb{N}.\leqno{\mathrm{(II_k)}}$$

In the sequel, we shall also use the following
\begin{lemma}\label{dalma}
	Let $\psi\in W^{1,p}(\Omega)$ and $u_0\in W^{1,p}(\Omega)$, with $p\ge 1$. Suppose that there exists $g\in u_0+ W^{1,p}_0(\Omega)$ such that
	$$g(x)\ge \psi(x)\qquad\quad \text{a.e. in }\quad\Omega.$$
	Then there exists a non increasing sequence of functions $\psi_k\in u_0+ W^{1,p}_0(\Omega)$ such that
	$$\psi_k\to \psi \qquad\quad \text{a.e. in }\quad\Omega.$$
\end{lemma}
For the proof we refer   to \cite[Lemma 2.19]{SS18} which is a suitable version for our purposes of \cite[Lemma 1.5]{DM}.

\subsection{Harmonic extension of Sobolev functions}
Let us fix a ball $B_R\subset \mathbb{R}^n$ and consider the following Dirichlet problem
\begin{equation*}
	\begin{cases}
		\Delta h=0\qquad\qquad\text{in}\,\,B_R\cr\cr
		u=f\qquad\qquad\quad\text{on}\,\,\partial B_R	\end{cases}\eqno{\mathrm{(D)}}
\end{equation*}
where $f\in W^{1,p}(\partial B_R)$. Browder-Minty Theorem implies that problem (D) admits a unique solution $u\in W^{1,p}(B_R)$ and we can define the solution operator
$$\mathcal{S}_{\Delta}:\,\, f\in W^{1,p}(\partial B_R)\mapsto u\in W^{1,p}(B_R)$$ 
We shall use the following particular case of \cite[Theorem 4.1]{CKPJmaa} 
\begin{theorem}\label{ext}
For $1<p\le q\le\frac{pn}{n-1}$, it holds that
$$||D \mathcal{S}_{\Delta}(f)||_{L^q(B_R;\mathbb{R}^n)}\le c(n,p,q)||Df||_{L^p(\partial B_R;\mathbb{R}^{n-1})}||f||_{L^p(\partial B_R)}$$
\end{theorem}
It is worth mentioning that previous result is well known, but it is difficult to find an explicit proof and this is the reason why we refer to \cite{CKPJmaa}. 

\subsection{Dual formulation of the obstacle problem}

This section is devoted to establish  the dual formulation of  obstacle problems with standard growth conditions, extending classical ideas of Kohn and Temam \cite{KT83} and Anzellotti \cite{A84} and following \cite{SS18}. 
For the readers' convenience  we recall a few key results about convex duality here and we refer to \cite[Chapter 1]{ET} for details.
\\
Given a convex continuous function $F:\mathbb{R}^n\to \mathbb{R}$,
its polar (or Fenchel conjugate)  is defined by
\begin{equation}
\label{Fconiug}
F^*(\zeta) := \sup_{\xi \in \mathbb{R}^n} \left (\langle \zeta, \xi \rangle - F(\xi) \right ) \qquad {\forall \,}\zeta \in \mathbb{R}^n.
\end{equation}
The function $F^*: \mathbb{R}^n \rightarrow \mathbb{R}$ is convex and, if $F$ satisfies assumption (H1), $F^*$ has $(q', p')$ growth, where $p'$ and $q'$ are the H\"older conjugate exponents of $p, q$ respectively, i.e. there exist constants $c(L), c(\ell)$ such that
$$c(L)|\zeta|^{q'}\le F^*(\zeta)\le c(\ell)|\zeta|^{p'} \qquad {\forall \,\zeta \in \mathbb{R}^n.}$$
\\
One can check that the bipolar integrand $F^{**} := (F^*)^*$ equals $F$ at $\xi$ if and only if $F$ is lower semicontinuous and convex at $\xi$, as it is the case here.

From the definition of polar function directly follows  the Young-type (or Fenchel) inequality
\begin{equation}
\label{Fenchel}
\langle \zeta, \xi \rangle \le \, F^*(\zeta) + F^{**}(\xi)
\end{equation}
for all $\zeta, \xi \in \mathbb{R}^n$.
\\
 Notice that, for a given $\xi$, we have equality in \eqref{Fenchel} precisely for $\zeta \in \partial F^{**}(\xi)$, the subgradient of $F^{**}$ at $\xi$. In particular, when $F$ is  $\mathcal{C}^1$, 
 for every $\xi\in\mathbb{R}^n$, we have equality in \eqref{Fenchel} precisely for $\zeta =F'(\xi)$. Actually, it holds the following 
\begin{equation}
\label{Fenchel-equal}
F(\xi) + F^*(F'(\xi)) = \langle F'(\xi), \xi \rangle,
\end{equation}
for every $\xi\in\mathbb{R}^n$.
Indeed, the convexity of $F$, since $F\in \mathcal{C}^1$, gives
\[
F(\xi) \ge \, F(\eta) + \langle F'(\eta), \xi - \eta \rangle \qquad \forall\,\, \xi, \eta \in \mathbb{R}^n,
\]
which is equivalent to 
\[
\langle F'(\eta), \eta \rangle - F(\eta) \ge \, \langle F'(\eta), \xi \rangle - F(\xi) \qquad \forall\,\, \xi, \eta \in \mathbb{R}^n. 
\]
From this we deduce that
\[
\langle F'(\eta), \eta \rangle - F(\eta) \ge \, \sup_{\xi \in \mathbb{R}^n} \left[ \langle F'(\eta), \xi \rangle - F(\xi) \right ] = F^*(F'(\eta)), 
\]
by the definition of $F^*(\zeta)$ at \eqref{Fconiug}.
Thus
\[
\langle F'(\eta), \eta \rangle \ge \, F(\eta) + F^*(F'(\eta)) 
\]
which obviously gives equality in \eqref{Fenchel}, i.e. \eqref{Fenchel-equal}.

Now, we consider for any $p > 1$
\begin{equation}
\label{classeSpprimo}
S^{p'}_{-}(\Omega) = \{\sigma \in L^{p'}(\Omega): \,\, {\rm div} \sigma \le \, 0\quad \text{in}\quad \mathcal{D}'(\Omega)\},
\end{equation}
where as usual $p'=\frac{p}{p-1}$ and, for $u_0, U \in W^{1,p}(\Omega),$  we introduce a measure  [\![$\sigma, DU$]\!]$_{u_0}$ on $\overline{\Omega}$ by setting
\begin{equation}
\label{duality}
[\![\sigma, DU]\!]_{u_0} ( \overline{\Omega}) = \int_{\Omega}  (U - u_0) \, d (- {\rm div} \sigma)  + \int_{\Omega} \langle \sigma, D u_0 \rangle \, dx.
\end{equation}
For $\tilde{u} \in u_0 + W^{1,p}_0(\Omega)$,  the measure $[\![\sigma, D\tilde{u}]\!]_{u_0}(\bar\Omega)$ corresponds to the function $\langle \sigma, D \tilde{u} \rangle \in L^1(\Omega)$ as it follows from the well known integration by parts formula
\begin{equation}
\label{by-parts}
\int_{\Omega} \varphi d (- {\rm div} \sigma) = \int_{\Omega} \langle \sigma, D \varphi \rangle \, dx,
\end{equation} 
for every $\varphi\in W^{1,p}_0(\Omega)$.
 
The dual formulation of  obstacle problems under standard growth conditions  is contained in the following:
\begin{theorem}
\label{teo1}
Let $G: \mathbb{R}^n \rightarrow \mathbb{R}$ be a $\mathcal{C}^1$, strictly convex function satisfying
\[
\ell_p \,( |\xi|^p-1) \le \, G(\xi) \le \, L_p(1 + |\xi|^p),
\]
for all $\xi \in \mathbb{R}^n$ and an exponent $p > 1$. Then
\begin{equation}
\label{dual}
\min_{v \in {\mathcal{K}_{\psi}(\Omega)}} \int_{\Omega} G(Dv) \, dx = \max_{\sigma \in S^{p'}_{-}(\Omega)} \left ([\![\sigma, D\psi]\!]_{u_0}(\overline{\Omega}) - \int_{\Omega} G^*(\sigma) \, dx \right )
\end{equation}
where  $S^{p'}_{-}(\Omega)$,   $[\![\sigma, D\psi]\!]_{u_0}$  and ${\mathcal{K}_{\psi}}(\Omega)$ are defined   in \eqref{classeSpprimo},  \eqref{duality}  and \eqref{classeA} respectively.
\\
If moreover $u\in \mathcal{K}_\psi(\Omega)$ is the solution to \eqref{p-obst-def0}, then
\begin{equation}
\label{dualminmax}
\int_{\Omega} G(Du) \, dx = [\![G'(Du), D\psi]\!]_{u_0}(\overline{\Omega}) - \int_{\Omega} G^*(G'(Du)) \, dx. 
\end{equation}
\end{theorem}
\begin{proof}
Let us consider an arbitrary vector field $\sigma \in S^{p'}_{-}(\Omega)$ and a function $v \in \mathcal{K}_{\psi}(\Omega)$. Since $- {\rm div} \sigma$ is a non-negative Radon measure and $v \ge \psi$ in $\Omega$ a.e. in $\Omega$, we have
\begin{equation}
\label{var-ineq}
\int_{\Omega} (v - \psi) d (- {\rm div} \sigma) \ge \, 0.
\end{equation}
By the definition at \eqref{duality} and by using \eqref{var-ineq}, we infer that
\begin{eqnarray*}
[\![\sigma, D\psi]\!]_{u_0}(\overline{\Omega}) &=& \int_{\Omega} (\psi - u_0) d (- {\rm div} \sigma)  + \int_{\Omega} \langle \sigma, D u_0 \rangle \, dx \\
 &=& \int_{\Omega} (\psi - v + v - u_0) d (- {\rm div} \sigma)  + \int_{\Omega} \langle \sigma, D u_0 \rangle \, dx \\
 &\le& \int_{\Omega} (v - u_0) d (- {\rm div} \sigma)  + \int_{\Omega} \langle \sigma, D u_0 \rangle \, dx.
\end{eqnarray*}
Since $v,u_0\in W^{1,p}(\Omega)$ and $v = u_0$ on $\partial \Omega$, we can use \eqref{by-parts} in the first integral of the last line of previous formula, thus getting
\begin{eqnarray*}
[\![\sigma, D\psi]\!]_{u_0}(\overline{\Omega}) &=& \int_{\Omega} \langle \sigma, Dv - Du_0 \rangle \, dx  + \int_{\Omega} \langle \sigma, D u_0 \rangle \, dx \\
&=& \int_{\Omega} \langle \sigma, Dv \rangle \, dx \le \, \int_{\Omega} G(Dv) \, dx + \int_{\Omega} G^*(\sigma) \, dx,
\end{eqnarray*}
by the Young's type inequality at {\eqref{Fenchel}} and since, by our assumptions on $G$, it holds that $G=G^{**}$.
 This entails that
\[
\int_{\Omega} G(Dv) \, dx \ge \, [\![\sigma, D\psi]\!]_{u_0}(\overline{\Omega}) - \int_{\Omega} G^*(\sigma) \, dx
\]
and thus, passing to the minimum in $v$ in the left hand side  and to the maximum in $\sigma$ in the right hand side of previous inequality, we get
\begin{equation}\label{minmax}
	\min_{v \in {\mathcal{K}_{\psi}(\Omega)}} \int_{\Omega} G(Dv) \, dx \ge \, \max_{\sigma \in S^{p'}_{-}(\Omega)} \left ([\![\sigma, D\psi]\!]_{u_0}(\overline{\Omega}) - \int_{\Omega} G^*(\sigma) \, dx \right ).
\end{equation}
In order to prove the reverse inequality, we start from \eqref{Fenchel-equal} with the choice $F=G$ and $\xi = Du$, being $u \in {\mathcal{K}_{\psi}}$ a solution to \eqref{p-obst-def0}. Integrating over $\Omega$ we  have that
\begin{eqnarray*}
\int_{\Omega} G(Du) \, dx &=&  \int_{\Omega} \langle G'(Du), Du \rangle \, dx - \int_{\Omega} G^*(G'(Du)) \, dx \\
&=& \int_{\Omega} \langle G'(Du), Du - Du_0 \rangle \, dx + \int_{\Omega} \langle G'(Du), D u_0 \rangle \, dx - \int_{\Omega} G^*(G'(Du)) \, dx. 
\end{eqnarray*}
Accordingly to the terminology used so far, we set $\sigma := G'(Du)$ and recall that, since $G$ satisfies standard growth condition, the  variational inequality
at \eqref{var-ineq-varphi}
	 holds true for every $\varphi\in \mathcal{C}^{\infty}_0(\Omega), \varphi \ge 0$ and therefore $\sigma\in S_-^{p'}(\Omega)$. Then we have
\begin{eqnarray*}
\int_{\Omega} G(Du) \, dx &=& \int_{\Omega} \langle \sigma, Du - Du_0 \rangle \, dx + \int_{\Omega} \langle \sigma, D u_0 \rangle \, dx - \int_{\Omega} G^*(\sigma) \, dx\\
&=& \int_{\Omega} (u - u_0) \, d (- {\rm div} \sigma) + \int_{\Omega} \langle \sigma, D u_0 \rangle \, dx - \int_{\Omega} G^*(\sigma) \, dx\\
&=& \int_{\Omega} (\psi - \psi + u - u_0) \, d (- {\rm div} \sigma) + \int_{\Omega} \langle \sigma, D u_0 \rangle \, dx - \int_{\Omega} G^*(\sigma) \, dx\\
&=& \int_{\Omega} (\psi - u_0) \, d (- {\rm div} \sigma) - \int_{\Omega} (\psi - u) \, d (- {\rm div} \sigma) + \int_{\Omega} \langle \sigma, D u_0 \rangle \, dx - \int_{\Omega} G^*(\sigma) \, dx\\
&=& [\![\sigma, D\psi]\!]_{u_0}(\overline{\Omega}) - \int_{\Omega} \langle \sigma, D \psi - Du \rangle \, dx - \int_{\Omega} G^*(\sigma) \, dx\\
&\le& [\![\sigma, D\psi]\!]_{u_0}(\overline{\Omega}) - \int_{\Omega} G^*(\sigma) \, dx,
\end{eqnarray*}
where we used once more the integration by parts formula \eqref{by-parts} and in the last line we exploited the fact that
\begin{equation}\label{var2}
\int_{\Omega} \langle \sigma, D \psi - Du \rangle \, dx \ge \, 0.	
\end{equation}

Actually, since $u\in u_0+W^{1,p}(\Omega)$ and $u\ge \psi$ as long as $u$ is the solution of our obstacle problem, we can use Lemma \ref{dalma} to deduce that there exists a non increasing sequence $\psi_k\in u_0+W^{1,p}(\Omega)$ such that $$\psi_k\to \psi\qquad \text{a.e. in}\,\,\, \Omega.$$
Therefore, 
$$\psi_k\in \mathcal{K}_\psi(\Omega)$$
and, since $u$ is a solution of \eqref{var-ineq-stand}, we have
\[
\int_{\Omega} \langle \sigma, D \psi_k - Du \rangle \, dx \ge \, 0 \qquad \text{for every}\,\, k\in \mathbb{N}. 
\]
Via the monotone convergence Theorem, passing to the limit in previous inequality, we deduce the validity of \eqref{var2}.

Summing up we have
\begin{eqnarray*}
\min_{v \in {\mathcal{K}_{\psi}(\Omega)}} \int_{\Omega} G(Dv) \, dx &=& \int_{\Omega} G(Du) \, dx \\
&\le& \, [\![\sigma, D\psi]\!]_{u_0}(\overline{\Omega}) - \int_{\Omega} G^*(\sigma) \, dx \\
&\le& \, \max_{\sigma \in S^{p'}_{-}(\Omega)}\left \{  [\![\sigma, D\psi]\!]_{u_0}(\overline{\Omega}) - \int_{\Omega} G^*(\sigma) \, dx\right \}.	
\end{eqnarray*}
Combining previous estimate with \eqref{minmax} we establish \eqref{dual} and, recalling that $\sigma=G'(Du)$, the equality at \eqref{dualminmax}.
\end{proof}
\medskip
\section{Proof of Theorem \ref{teo2}}
In this section we shall establish the validity of the variational inequality associated to our obstacle problem, by using the duality theory and the approximation Lemma of the previous Section. More precisely we are ready to give the
\begin{proof}[Proof of Theorem \ref{teo2}]
For the sake of clarity we shall divide the proof in steps. In the first one, we shall use the approximation Lemma to construct a sequence of obstacle problems with standard growth condition for which the dual problem is given by Theorem \ref{teo1}. In the second step, we prove that the sequence of approximating minimizers converges to the solution of problem \eqref{obst-def0}, as well as the sequence of dual maximizers converges to a field whose divergence is a non positive Radon measure. Finally in Step 3 and 4 we establish the validity of the variational inequality.
 
\medskip
{\it Step 1. The approximation}\quad
Let $F_k$ be the sequence of functionals obtained applying  Lemma \ref{lemmaCGM} to the integrand $F$. We recall that $F_k \nearrow F$ and that $F_k$ are of class $\mathcal{C}^1$ and strictly convex, with $p-$growth.
\\
Let $u_k \in {\mathcal{K}_{\psi}(\Omega)}$ be the solution to the obstacle problem
\begin{equation}
\label{minFk}
\min_{w \in \mathcal{K}_{\psi}(\Omega)} \int_{\Omega} F_k(Dw) \, dx
\end{equation}
and let 
\begin{equation}
\label{sigmak}
\sigma_k := F'_k(Du_k) \in \mathcal{S}_-^{p'}(\Omega)
\end{equation}
be the solution of the dual problem given by \eqref{dual}, i.e. $\sigma_k$ is such that
\[
\max_{\sigma \in S^{p'}_{-}(\Omega)} \left \{ [\![\sigma, D\psi]\!]_{u_0}(\overline{\Omega}) - \int_{\Omega} F_k^*(\sigma)\,dx \right \} = [\![\sigma_k, D\psi]\!]_{u_0}(\overline{\Omega}) - \int_{\Omega} F_k^*(\sigma_k)\,dx, 
\]
where $F^*_k$ denotes the polar function of $F_k$. By (II) and (III) of Lemma \ref{lemmaCGM}, we are legitimate to apply Theorem \ref{teo1} to each $F_k$. Therefore, from \eqref{dualminmax} with $G = F_k$, $u = u_k$ and $\sigma = \sigma_k$, we have that the following equality
\[
\int_{\Omega} F_k(Du_k) \, dx = [\![\sigma_k, D\psi]\!]_{u_0}(\overline{\Omega}) - \int_{\Omega} F_k^*(\sigma_k) \, dx 
\]
holds for all $k \in \mathbb{N}$.
\\
As long as $F_k$ satisfy a uniform $(p,q)-$growth condition, then, as already remarked in Section \ref{prelim}, $F_k^*$ satisfy a uniform $(q', p')-$growth condition, and, since $F_k(\xi)\nearrow F(\xi)$, it is not difficult to check that $F_k^*(\zeta) \searrow F^*(\zeta)$ as $k \rightarrow \infty$, pointwise in $\zeta$.  Furthermore, since $F_k$ satisfy standard growth conditions, we also have that $u_k$ solve the following variational inequality 
\begin{equation}
\label{var-ineq2}
\int_{\Omega} \langle \sigma_k, D \varphi - D u_k \rangle \, dx \ge \, 0 \qquad \forall\, \varphi \in \mathcal{K}_{\psi}(\Omega)\quad \text{and}\,\,\forall\, k\in\mathbb{N}.
\end{equation}

\medskip
{\it Step 2. Passage to the limit.}\quad
Our next purpose is  to prove that $u_k \rightarrow u$ strongly in $W^{1,p}(\Omega)$, where $u$ is the solution of the obstacle problem at \eqref{obst-def0}.

First of all, we observe that
\begin{eqnarray}\label{crescita1}
	L_0 \int_{\Omega} |D u_k|^p \, dx &\le& \int_{\Omega} [F_k(D u_k) + L_0] \, dx\le \, \int_{\Omega} [F_k(Du_0) + L_0] \, dx\cr\cr
	& \le &  \, \int_{\Omega} [F(Du_0) + L_0] \, dx < + \infty,
\end{eqnarray}

where we  used \eqref{assumuzero}, the growth condition on $F_k$ expressed at (III) of Lemma \ref{lemmaCGM}, the minimality of $u_k$, the fact that, by virtue of Remark \ref{rem-serena}, we can use $u_0$ as test function and also that $F_k \nearrow F$. This tells us that the sequence $\{u_k\}_{k}$ is bounded in $W^{1,p}(\Omega)$. Then, by the reflexivity of $W^{1,p}(\Omega)$, it admits a subsequence weakly converging to some $v \in W^{1,p}(\Omega)$. We have that $v \in \mathcal{K}_{\psi}(\Omega)$ because $u_k \in \mathcal{K}_{\psi}(\Omega)$ and $\mathcal{K}_{\psi}(\Omega)$ is a convex  closed set, therefore weakly closed. 

Fix $k_0\in \mathbb{N}$, by the lower semicontinuity of $ F_{k_0}$, we have
\[
\liminf_{k \rightarrow + \infty} \int_{\Omega} F_{k_0}(D u_k) \, dx \ge \, \int_{\Omega} F_{k_0}(D v) \, dx 
\]
and the monotonicity of the sequence $F_k$ yields 
\[\int_{\Omega} F_{k_0}(D u_k) \, dx\le \int_{\Omega} F_{k}(D u_k) \, dx\]
for every $k>k_0$. Therefore
\[
 \, \int_{\Omega} F_{k_0}(D v) \, dx 
\le \liminf_{k \rightarrow + \infty} \int_{\Omega} F_{k}(D u_k) \, dx\]
and since $F_k \nearrow F$, taking the limit as $k_0\to \infty$ , we deduce, by the monotone convergence theorem, that 
\begin{equation}
\label{ineq1}
\liminf_{k \rightarrow + \infty} \int_{\Omega} F_k(D u_k) \, dx \ge \, \int_{\Omega} F(D v) \, dx.
\end{equation}
Thus in particular it turns out that $v \in \mathbb{K}_{\psi}^F(\Omega)$ and so we can exploit the minimality of $u$ in the class $\mathbb{K}_{\psi}^F(\Omega)$ to finally end up with
\begin{equation}
\label{ineq1-bis}
\liminf_{k \rightarrow + \infty} \int_{\Omega} F_k(D u_k) \, dx \ge \, \int_{\Omega} F(D v) \, dx \ge \, \int_{\Omega} F(D u) \, dx.
\end{equation}
On the other hand, by the minimality of $u_k$ we have
$$\int_{\Omega} F_k(D u_k) \, dx \le \int_{\Omega} F_k(D u) \, dx,$$
since $u\in \mathbb{K}_\psi(\Omega)\subset \mathcal{K}_\psi(\Omega).$  Using once more  the monotone convergence theorem, we get
\begin{equation}
\label{ineq2}
\limsup_{k \rightarrow + \infty} \int_{\Omega} F_k(D u_k) \, dx \le \, \limsup_{k \rightarrow + \infty} \int_{\Omega} F_k(D u) \, dx = \int_{\Omega} F(D u) \, dx.
\end{equation}
By a direct comparison of \eqref{ineq1} and \eqref{ineq2} we deduce that
\begin{equation}
\label{ineq3}
\int_{\Omega} F_k(D u_k) \, dx \rightarrow \int_{\Omega} F(Du) \, dx = \int_{\Omega} F(D v) \, dx,
\end{equation}
but the strict convexity of $F$ implies the uniqueness of the solutions and therefore $u = v$. 

To deduce the strong convergence of $u_k$ towards $u$, we exploit (II) of Lemma \ref{lemmaCGM} and \eqref{var-ineq2} with $u$ in place of $\varphi$, namely 
\begin{eqnarray*}
c(p, \nu)\int_{\Omega}  \, |V_p(Du) - V_p(Du_k)|^2 \, dx &\le& \, \int_{\Omega} \left (F_k(Du) - F_k(D u_k) - \langle F_k'(Du_k), Du - D u_k \rangle \right ) \, dx \\
&\le& \, \int_{\Omega} (F_k(Du) - F_k(D u_k)) \, dx \rightarrow 0\\
\end{eqnarray*}
as $k \rightarrow + \infty$, where in the last line we used \eqref{ineq3} and the equality in \eqref{ineq2}. Therefore, by Lemma \ref{Vi}, we get
\[
{c}\int_{\Omega}  \, |Du - Du_k|^2(1+|Du|^2+|Du_k|^2)^{\frac{p-2}{2}} \, dx \le \, \int_{\Omega} |V_p(Du) - V_p(Du_k)|^2 \, dx \rightarrow 0
\]
which entails the desired strong convergence  $$u_k \rightarrow u\qquad \textnormal{strongly in $W^{1,p}{(\Omega)}$}.$$
Indeed, if $p\ge 2$, this follows from the trivial inequality
$$|Du-Du_k|^p\le |Du-Du_k|^2(1+|Du|^2+|Du_k|^2)^{\frac{p-2}{2}}$$
while, for $1<p<2$, we may use H\"older's inequality with exponents $\frac{2}{p}$ and $\frac{2}{2-p}$ as follows
$$\int_{\Omega}  \, |Du - Du_k|^p \, dx=\int_{\Omega}  \, |Du - Du_k|^p(1+|Du|^2+|Du_k|^2)^{\frac{p(p-2)}{4}}(1+|Du|^2+|Du_k|^2)^{\frac{p(2-p)}{4}} \, dx$$
$$\le \left(\int_{\Omega}  \, |Du - Du_k|^2(1+|Du|^2+|Du_k|^2)^{\frac{(p-2)}{2}}\,dx\right)^{\frac{p}{2}}\left(\int_{\Omega}(1+|Du|^2+|Du_k|^2)^{\frac{p}{2}}  \,dx\right)^{\frac{2-p}{2}}$$
$$\le C\left(\int_{\Omega}  \, |Du - Du_k|^2(1+|Du|^2+|Du_k|^2)^{\frac{(p-2)}{2}}\,dx\right)^{\frac{p}{2}},$$
where, in the last line, we used \eqref{crescita1}.


The use of (IV) of Lemma \ref{lemmaCGM} for $\xi_k = Du_k$ and $\xi = Du,$ yields that 
$$\sigma_k  = F'_k(Du_k) \rightarrow F'(Du)\qquad  \text{locally uniformly {as $k \rightarrow \infty$.}}$$
It follows in particular that $F'_k(Du_k) \rightarrow F'(Du)$ in measure on $\Omega$ and so passing to the limit in the equality
\begin{equation}
\label{Fenchelk}
\langle \sigma_k, Du_k \rangle = F_k^*(\sigma_k) + F_k(Du_k),
\end{equation}
which has been deduced by \eqref{Fenchel-equal} with $\xi = Du_k$ and $G = F_k$, we recover, with $\sigma = F'(Du),$ the pointwise extremality relation 
\begin{equation}
\label{Fenchel-lim}
\langle F'(Du), Du \rangle = F^*(F'(Du)) + F(Du).
\end{equation}
\\
{\it Step 3. The validity of \eqref{tuttoL1}.}
\quad
Integrating \eqref{Fenchelk} over $\Omega$ and since we have assumed, by virtue of Remark \ref{rem-serena}, that $u_0 \in \mathbb{K}_{\psi}^{F}(\Omega),$ we obtain
\begin{eqnarray*}
\int_{\Omega} F_k^*(\sigma_k) \, dx &=& \int_{\Omega} \langle \sigma_k, D u_k \rangle \, dx - \int_{\Omega} F_k(Du_k) \, dx \\
&\le& \, \int_{\Omega} \langle \sigma_k, D u_0 \rangle \, dx  - \int_{\Omega} F_k(Du_k) \, dx \\
&=& \frac{1}{2} \int_{\Omega} \langle \sigma_k, 2 D u_0 \rangle \, dx  - \int_{\Omega} F_k(Du_k) \, dx \\
&\le & \frac{1}{2} \int_{\Omega} F_k^*(\sigma_k) \, dx + \frac{1}{2} \int_{\Omega} F_k(2 D u_0)  - \int_{\Omega} F_k(Du_k) \, dx.
\end{eqnarray*}
Reabsorbing the first term in the right hand side by the left hand side, we get
\begin{equation}
\label{controlloFstark}
\frac{1}{2} \int_{\Omega} F_k^*(\sigma_k) \, dx \le \, \frac{1}{2} \int_{\Omega} F(2 D u_0) \, dx  - \int_{\Omega} F_k(Du_k) \, dx \le \, C \int_{\Omega} F(D u_0) \, dx,
\end{equation}
by (H3) and \eqref{crescita1}. 

Recalling that $F^*_k \searrow F^*,$ from \eqref{controlloFstark} we also have that
\begin{equation*}
\label{3.11bis}
\int_{\Omega} F^*(\sigma_k) \, dx \le \, C \, \int_{\Omega} F(D u_0) \, dx
\end{equation*}

Since we already observed that $\sigma_k \rightarrow F'(Du)$ locally uniformly 
and $F^*(\sigma_k)\ge 0$ for every $k$, by Fatou's lemma and by previous estimate
\[
\int_{\Omega} F^*(F'(Du)) \, dx \le \, \liminf_{k \rightarrow + \infty} \int_{\Omega} F^*(\sigma_k) \, dx \le \, C \, \int_{\Omega} F(Du_0) \, dx.
\]
Thus 
\[
F^*(F'(Du)) \in L^1(\Omega).
\]
Whence, by \eqref{Fenchel-lim}, we also have
\[
\langle F'(Du), Du \rangle \in L^1(\Omega).
\]
since $F(Du)\in L^1(\Omega)$ by the definition of minimizer. 

\medskip

{\it Step 4.The validity of the variational inequality.}
\\
For this purpose, we note that in view of the  $(q', p')-$growth of $F^*(	\sigma)$ and of  $F^*_k(\sigma_k)$, previous  inequality and \eqref{controlloFstark},
\begin{equation}
\label{sigmaLq'}
\int_{\Omega} |F'(Du)|^{q'} \,dx \le \limsup_{k\to +\infty}  \int_{\Omega} |\sigma_k|^{q'}\le  \, C \, \int_{\Omega} F(Du_0) \, dx.
\end{equation}
Therefore $$\sigma_k\rightharpoonup \sigma\qquad \text{weakly\,\,in}\qquad L^{q'}(\Omega)$$
and by the convergence of $\sigma_k$  to $\sigma$ in measure, we also have
\begin{equation}\label{strong}
\sigma_k\to \sigma\qquad \text{strongly\,\,in}\qquad L^{r}(\Omega)\qquad \text{for\,every}\,\, r<q'
\end{equation}
and therefore also $\sigma_k\to \sigma$ a.e. up to a subsequence.
The minimality of $u_k$ yields the validity of the following variational inequality
\begin{equation}\label{conv1}
\int_\Omega \langle \sigma_k,D\eta\rangle\,dx\ge 0\qquad \text{for\, all}\,\, \eta\in  C^\infty_0(\Omega),\,\,\,\eta\ge 0	,
\end{equation}
and so, by the weak convergence of $\sigma_k$ to $\sigma$ in $L^{q'}(\Omega)$, passing to the limit as $k\to \infty$ in previous inequality, also
 \begin{equation}\label{conv12}
\int_\Omega \langle \sigma,D\eta\rangle\,dx\ge 0\qquad \text{for\, all}\,\, \eta\in  C^\infty_0(\Omega),\,\,\,\eta\ge 0	,
\end{equation}
This yields that $\mathrm{div}\,\sigma\le 0$ in the distributional sense, i.e. \eqref{div}. 
By \eqref{var-ineq2}, we have
\begin{equation}\label{conv1}
\int_\Omega \langle \sigma_k,Dz-Du_k\rangle\,dx\ge 0\qquad \text{for\, all}\,\, z\in  \mathbb{K}_\psi^F(\Omega),
\end{equation}
since $\mathbb{K}_\psi^F(\Omega)\subset \mathcal{K}_\psi(\Omega)$.
Before going on, we note that 
\begin{equation}\label{conv2}
	\int_\Omega \langle \sigma,Du\rangle\,dx\le \liminf_{k\to +\infty}\int_\Omega \langle \sigma_k,Du_k\rangle\,dx.
\end{equation}
Indeed, by \eqref{Fenchelk} we get
\begin{equation}
\langle \sigma_k,Du_k\rangle=F^*_k(\sigma_k)+F_k(Du_k)\ge C(L)|\sigma_k|^{q'}+	L_0|Du_k|^p-\mu_k\ge -\mu_k,
\end{equation}
where we used (II$_k$) and that $F^*_k(\xi)\ge F^*(\xi)\ge C(L)|\xi|^{q'}$. 
\\
Therefore
for the nonnegative sequence of functions $\langle \sigma_k,Du_k\rangle +\mu_k\ge 0$ that converges a.e. to $\langle \sigma,Du\rangle$, we are legitimate to apply Fatou's Lemma to deduce that
$$\int_\Omega \langle \sigma,Du\rangle\,dx\le \liminf_{k\to +\infty}\int_\Omega \big(\langle \sigma_k,Du_k\rangle+\mu_k\big)\,dx=\liminf_{k\to +\infty}\int_\Omega \langle \sigma_k,Du_k\rangle\,dx,$$
i.e. \eqref{conv2}.
Using \eqref{Fenchel}, we have  that
\[
|\langle \sigma_k,Dz\rangle| \le \, 2F_k^*(\sigma_k)+F(Dz)+F(-Dz),
\]
where we used that $F_k\le F$ for every $k\in \mathbb{N}$.
Therefore, by assumption \eqref{pm} and \eqref{controlloFstark}, the sequence $\langle\sigma_k,Dz\rangle$ is equi-integrable. In fact
\begin{eqnarray*}
\int_\Omega |\langle\sigma_k,Dz\rangle|\,dx &\le&  \int_\Omega F^*_k(\sigma_k)+\int_\Omega F(Dz)\,dx+\int_\Omega F(-Dz)\,dx\cr\cr
&\le& C\left(\int_\Omega F(Du_0)\,dx+\int_\Omega F(Dz)\,dx+\int_\Omega F(-Dz)\,dx\right).
\end{eqnarray*}
Using that  $\langle\sigma_k,Dz\rangle\to \langle\sigma,Dz\rangle$ a.e., Vitali's convergence Theorem implies
\begin{equation}\label{conv3bis}
	\langle\sigma_k,Dz\rangle\to \langle\sigma,Dz\rangle\qquad\text{strongly\,\,in}\qquad L^1(\Omega).
\end{equation}
Writing \eqref{conv1} as follows
$$\int_\Omega \langle \sigma_k,Dz\rangle\,dx\ge \int_\Omega \langle \sigma_k,Du_k\rangle\,dx $$
and taking the liminf as $k\to+\infty$ in previous equality, we get
$$\liminf_{k\to+\infty}\int_\Omega \langle \sigma_k,Dz\rangle\,dx\ge \liminf_{k\to+\infty}\int_\Omega \langle \sigma_k,Du_k\rangle\,dx $$
and so, using \eqref{conv3bis} in the left hand side and  \eqref{conv2} in the right hand side,
we conclude that
 \begin{equation*}
 \int_\Omega \langle \sigma,Dz\rangle\,dx\ge \int_\Omega \langle \sigma,Du\rangle\,dx \qquad \text{for\, all}\,\, z\in \mathbb{K}_\psi^F(\Omega)\,\, \text{such that}\,\, F(\pm Dz)\in L^1(\Omega),	
 \end{equation*}
 i.e. \eqref{pmlim}. 

\end{proof}
\section{Proof of Theorem \ref{teo3}}

This section is devoted to the proof of the regularity result stated in Theorem \ref{teo3}. 
\begin{proof}[Proof of Theorem \ref{teo3}]
 We start by observing that every $z \in \mathcal{K}_{\psi}(\Omega)$ belongs to $ W^{1,p}(\Omega)$. For every $x_0 \in \Omega$ we can find a ball $B = B(x_0, R) \subset \Omega$ such that $z_{|_{\partial B}} \in W^{1,p}(\partial B)$, see for instance {\cite{ziemer}}. Then,  by Theorem \ref{ext}, $z_{|_{\partial B}}$ has an harmonic extension $h$ to $B$, such that $h \in W^{1, \frac{np}{n-1}}(B)$. 
 Therefore, by the assumption  $q < \frac{np}{n-1}$ and by virtue of (H1), we get $F(Dh) \in L^1(B)$. 
 Moreover by the maximum principle, since $z \ge \psi$ a.e. in $B$ and $z = h$ on $\partial B$, also $h \ge \psi$ a.e. in $B$. Therefore $h \in {\mathbb{K}_{\psi}^F(B)}$ and the arguments of the proof of Theorem \ref{teo2}, replacing $\Omega$ with $B$ and $u_0$ with $h$, give that
 \[
F^*(F'(Du)) \in L^1_{\loc}(\Omega) \qquad \qquad \langle F'(Du), Du \rangle \in L^1_{\loc}(\Omega)
\]
and
$${\rm div} F'(Du) \le \, 0$$
locally, in the distributional sense.
\\Our next purpose is to prove that $u\in W^{1,q}_{\loc}(\Omega)$.
To this aim, let $F_k$, $u_k$ be respectively the sequence of functionals and their minimizers introduced in the proof of previous Theorem.  
Let us consider $\varphi_k := u_k + t v_k$ for a suitable $v_k \in W^{1,p}_0(\Omega)$ such that
\begin{equation}
\label{cond-v}
u_k - \psi + t \, v_k \ge 0 \qquad \textnormal{for $t \in [0,1)$}.
\end{equation}
Such function $\varphi_k$ belongs to the obstacle class ${\mathcal{K}_{\psi}(\Omega)}$, because $\varphi_k = u_k + t v_k \ge \psi$ and $\varphi_k \in u_0+ W^{1,p}_{\loc}(\Omega)$. 
\\
Now we fix  balls $B_{\frac{R}{2}}\subset B_\rho\subset B_{R}$ such that $B_{2R}\Subset \Omega$ and  a cut off function $\eta\in C^\infty_0(B_R)$, $\eta \equiv 1$ on $B_{\rho}$ such that  $|\nabla \eta|\le \frac{c}{R-\rho}$. Due to the local nature of our results, we suppose $R\le 1$ without loss of generality.
Then, for  $|h|<\frac{R}{4}$, we take
\begin{equation}
\label{fun-v}
v^1_k(x) = \eta^2(x) [(u_k - \psi)(x + h) - (u_k - \psi)(x)].
\end{equation}
 From the regularity of $u_k$ and $\psi$, we deduce that $v^1_k \in W^{1,p}_0(\Omega)$. Moreover $v^1_k$ fulfills \eqref{cond-v}. Indeed, for a.e. $x \in \Omega$ and for any $t \in [0,1)$
\begin{eqnarray*}
&&u_k(x)- \psi(x)  + t v^1_k(x)\\
&=& u_k(x) - \psi(x)  + t \eta^2(x) [(u_k - \psi)(x + h) - (u_k - \psi)(x)]\\
&=& t \, \eta^2(x) (u_k - \psi)(x+h) + (1 - t \eta^2(x)) (u_k - \psi)(x) \ge \, 0,
\end{eqnarray*}
because $u_k \in \mathcal{K}_{\psi}(\Omega)$.
\\
With this choice in  \eqref{var-ineq2}, we obtain 
\begin{equation}
\label{sei}
0 \le \, \int_{\Omega} \langle F_k'(Du(x)), D[\eta^2(x) [(u_k - \psi)(x + h) - (u_k - \psi)(x)]] \rangle \, dx.
\end{equation}
On the other hand, if we introduce
\begin{equation}
\label{fun-v-tras}
v^2_k(x) = \eta^2(x-h) [(u_k - \psi)(x-h) - (u_k - \psi)(x)]
\end{equation}
then  $v^2_k \in W^{1,p}_0(\Omega)$ and it satisfies condition \eqref{cond-v}, as long as
\begin{eqnarray*}
&&u_k(x)- \psi(x)  + t v^2_k(x)\\
&=& u_k(x) - \psi(x)  + t \eta^2(x-h) [(u_k - \psi)(x-h) - (u_k - \psi)(x)]\\
&=& t \, \eta^2(x-h) (u_k - \psi)(x-h) + (1 - t \eta^2(x-h)) (u _k- \psi)(x) \ge \, 0.
\end{eqnarray*}
Choosing in \eqref{var-ineq2} as test function $\varphi_k = u_k + t v^2_k$, where $v^2_k$ is defined in \eqref{fun-v-tras},
we get 
$$ 0 \le \, \int_{\Omega} \langle F'_k(Du_k(x)), D[\eta^2(x-h) [(u_k - \psi)(x-h) - (u_k - \psi)(x)]] \rangle \, dx,$$
Changing variable we get
\begin{equation}
\label{sette}
0 \le \, \int_{\Omega} \langle F'_k(Du_k(x+h)), D[\eta^2(x) [(u_k - \psi)(x) - (u_k - \psi)(x+h)]] \rangle \, dx.
\end{equation}
Thus by adding \eqref{sei} and \eqref{sette}, we obtain
\begin{eqnarray*}
0 &\le& \int_{\Omega} \langle F'_k(Du_k(x)), D[\eta^2(x) [(u_k - \psi)(x+h) - (u_k - \psi)(x)]] \rangle \, dx\\
&& +  \int_{\Omega} \langle F'_k(Du_k(x+h)), D[\eta^2(x) [(u_k - \psi)(x) - (u_k - \psi)(x+h)]] \rangle \, dx\\
&=& \int_{\Omega} \langle F'_k(Du(x)) - F'_k(Du_k(x+h)), D[\eta^2(x) [(u_k - \psi)(x+h) - (u_k - \psi)(x)]] \rangle \, dx,
\end{eqnarray*}
which implies
\begin{eqnarray*}
\!\!\!\! 0 &\ge& \int_{\Omega} \langle F'_k(Du_k(x+h)) - F'_k(Du_k(x)), \eta^2(x) D[(u_k - \psi)(x+h) - (u_k - \psi)(x)]\rangle \, dx\\
 && \!\!\!\!\!\! + \int_{\Omega} \langle F'_k(Du_k(x+h)) - F'_k(Du_k(x)), 2 \, \eta(x) \, D \eta(x) \, [(u_k - \psi)(x+h) - (u_k - \psi)(x)]\rangle \, dx.
\end{eqnarray*}
The previous inequality can be rewritten as follows
\begin{eqnarray}
0 &\ge & \int_{\Omega} \langle F'_k(Du_k(x+h)) - F'_k(Du_k(x)), \eta^2 (Du_k(x+h) - Du_k(x)) \rangle \, dx \nonumber\\
&& - \int_{\Omega} \langle F'_k(Du_k(x+h)) - F'_k(Du_k(x)), \eta^2 (D\psi(x+h) - D\psi(x)) \rangle \, dx \nonumber\\
&& + \int_{\Omega} \langle F'_k(Du_k(x+h)) - F'(Du_k(x)), 2 \eta \, D \eta \tau_{h}(u_k - \psi) \rangle \, dx \nonumber\\
&= :& I + II + III, \label{I-to-III}\end{eqnarray}
that yields
\begin{equation}\label{otto}
I \le \, |II| + |III|.
	\end{equation}
The ellipticity of $F_k$ expressed by (II) of Lemma \ref{lemmaCGM} implies
\begin{equation}\label{I}
I \ge \, c(p, \nu) \int_{\Omega} \eta^2 |\tau_{h} V_p(Du_k)|^2 \, dx.
\end{equation}
For the estimation of {$II$ and $III$,} we use H\"older's inequality  to deduce that
\begin{eqnarray}
\label{III}
{|II| + |III|} &\le& \,\left(\int_{\Omega} \eta^{2} |F_k'(Du_k(x))|^{q'} \, dx\right)^{\frac{1}{q'}}\left(\int_{\Omega}\eta^2|\tau_{h} D\psi|^q\, dx\right)^{\frac{1}{q}}\cr\cr
&&+\left(\int_{\Omega} \eta^{q'} |F'_k(Du_k(x))|^{q'} \, dx\right)^{\frac{1}{q'}}\left(\int_{\Omega}|D\eta|^q|\tau_{h} u_k|^q\, dx\right)^{\frac{1}{q}}\cr\cr
&&+	\left(\int_{\Omega} \eta^{q'} |F'_k(Du_k(x))|^{q'} \, dx\right)^{\frac{1}{q'}}\left(\int_{\Omega}|D\eta|^q|\tau_{h} \psi|^q\, dx\right)^{\frac{1}{q}}\cr\cr
&\le& \,\left(\int_{B_{\frac{3}{4}R}}  |F'_k(Du_k(x))|^{q'} \, dx\right)^{\frac{1}{q'}}\left(\int_{B_R}|\tau_{h} D\psi|^q\, dx\right)^{\frac{1}{q}}\cr\cr
&&+\frac{c}{R-\rho}\left(\int_{B_{\frac{3}{4}R}}  |F'_k(Du_k(x))|^{q'} \, dx\right)^{\frac{1}{q'}}\left(\int_{B_R}|\tau_{h} u_k|^q\, dx\right)^{\frac{1}{q}}\cr\cr
&&+\frac{c}{R-\rho}	\left(\int_{B_{\frac{3}{4}R}}  |F'_k(Du_k(x))|^{q'} \, dx\right)^{\frac{1}{q'}}\left(\int_{B_R}|\tau_{h} \psi|^q\, dx\right)^{\frac{1}{q}},
\end{eqnarray}
where we used the properties of $\eta$ and that, since $u\in W^{1,p}(\Omega)$ and $q<\frac{pn}{n-1}<\frac{pn}{n-p}$ we have that $\tau_h u\in L^q(\Omega)$.
We used also that$$ \int_{B_R}|f(x+h)|^{q'}\,dx\le c\int_{B_{\frac{3}{4}R}}|f(x)|^{q'}\,dx,$$
for $|h|<\frac{R}{4}$.
Denoted by $h_u$ the harmonic extension of $u$ to the ball $B_{2R}$, we can use  \eqref{sigmaLq'} with $h_u$ in place of $u_0$ to obtain
$$\left(\int_{B_{\frac{3}{4}R}}  |F'_k(Du_k(x))|^{q'} \, dx\right)^{\frac{1}{q'}}\le C \left(\int_{B_{2R}}  F(Dh_u(x)) \, dx\right)^{\frac{1}{q'}}$$
and so
\begin{eqnarray}
\label{II}
{|II| + |III|} &\le&\, \left(\int_{B_{2R}}  F(Dh_u(x)) \, dx\right)^{\frac{1}{q'}}\left(\int_{B_R}|\tau_{h} D\psi|^q\, dx\right)^{\frac{1}{q}}\cr\cr
&&+ \frac{C}{R-\rho}\left(\int_{B_{2R}}  F(Dh_u(x)) \, dx\right)^{\frac{1}{q'}}\left(\int_{B_R}|\tau_{h} u_k|^q\, dx\right)^{\frac{1}{q}}\cr\cr
&&+\frac{C}{R-\rho}	\left(\int_{B_{2R}}  F(Dh_u(x)) \, dx\right)^{\frac{1}{q'}}\left(\int_{B_R}|\tau_{h} \psi|^q\, dx\right)^{\frac{1}{q}}\cr\cr
&\le&\, {C} \left(\int_{B_{2R}}  {(|Dh_u(x)|^{q} + 1)}\, dx\right)^{\frac{1}{q'}}\left(\int_{B_R}|\tau_{h} D\psi|^q\, dx\right)^{\frac{1}{q}}\cr\cr
&&+ \frac{C}{R-\rho} \left(\int_{B_{2R}}  {(|Dh_u(x)|^{q} + 1)} \, dx\right)^{\frac{1}{q'}}\left(\int_{B_R}|\tau_{h} u_k|^q\, dx\right)^{\frac{1}{q}}\cr\cr
&&+\frac{C}{R-\rho}	\left(\int_{B_{2R}}  {(|Dh_u(x)|^{q} + 1)} \, dx\right)^{\frac{1}{q'}}\left(\int_{B_R}|\tau_{h} \psi|^q\, dx\right)^{\frac{1}{q}} 
\end{eqnarray}
where  we used the right inequality in (H1). Therefore, plugging \eqref{I} and \eqref{II} in \eqref{otto} and using  that $\eta\equiv 1$ on $B_{\rho}$, we obtain
\begin{eqnarray}\label{intermedia}
&&c(p, \nu) \int_{B_{\rho}}  |\tau_{h} V_p(Du_k)|^2\le  {C}
\left(\int_{B_{2R}}  {(|Dh_u(x)|^{q} + 1)} \, dx\right)^{\frac{1}{q'}}\left(\int_{B_R}|\tau_{h} D\psi|^q\, dx\right)^{\frac{1}{q}}
\cr\cr
&&+ \frac{C}{R}	\left(\int_{B_{2R}}  {(|Dh_u(x)|^{q} + 1)}\, dx\right)^{\frac{1}{q'}}\left(\int_{B_R}|\tau_{h} u_k|^q\, dx\right)^{\frac{1}{q}}\cr\cr
&&+	\frac{C}{R} \left(\int_{B_{2R}}  {(|Dh_u(x)|^{q} + 1)} \, dx\right)^{\frac{1}{q'}}\left(\int_{B_R}|\tau_{h} \psi|^q\, dx\right)^{\frac{1}{q}}. 
\end{eqnarray}
Using the assumption on $\psi$ and the embedding of Theorem \ref{thtartar} with
$$\alpha=1-n\left(\frac{1}{p}-\frac{1}{q}\right)\qquad\text{and}\qquad q=\frac{np}{n-\alpha p},$$ we arrive at
\begin{eqnarray}\label{intermedia2}
&& \int_{B_{R/2}}  |\tau_{h} V_p(Du_k)|^2 \cr\cr
&\le&  C|h|
\left(\int_{B_{2R}}  {(|Dh_u(x)|^{q} + 1)} \, dx\right)^{\frac{1}{q'}}\left(\int_{B_{2R}}| D^2\psi|^q\, dx\right)^{\frac{1}{q}}
\cr\cr
&&+ \frac{C|h|^\alpha}{R-\rho}	\left(\int_{B_{2R}}  {(|Dh_u(x)|^{q} + 1)}\, dx\right)^{\frac{1}{q'}}\left(\int_{B_R}|D u|^p\, dx\right)^{\frac{1}{p}}\cr\cr
&&+	\frac{C|h|}{R-\rho} \left(\int_{B_{2R}}  {(|Dh_u(x)|^{q} + 1)} \, dx\right)^{\frac{1}{q'}}\left(\int_{B_{2R}}|D \psi|^q\, dx\right)^{\frac{1}{q}}\cr\cr
&\le& C|h|^\alpha  \left(\int_{B_{2R}} {(|Dh_u(x)|^{q} + 1)} \, dx\right)^{\frac{1}{q'}}\left(||Du||_{L^p(B_{2R})}+||\psi||_{W^{2,q}(B_{2R})}\right). 
\end{eqnarray}
 Estimate \eqref{intermedia2} implies that $V_p(Du_k)\in B^{\frac{\alpha}{2},2}_\infty$ and therefore again by Lemma \ref{lefrac}, we have  $$Du_k\in L^t_{\loc}(\Omega)\qquad \text{for every}\quad t<\frac{np}{n\left(1+\frac{1}{p}-\frac{1}{q}\right)-1}$$
Following \cite{CKPPisa}, we now define the increasing sequence of exponents
$$p=p_0\qquad\qquad p_j=\frac{np}{n\left(1+\frac{1}{p_{j-1}}-\frac{1}{q}\right)-1}$$
One can easily check that
$$p_j\nearrow \frac{n(p-1)}{n-1-\frac{n}{q}}$$
and that for $q< \frac{pn}{n-1}$ we have
$$q< \frac{n(p-1)}{n-1-\frac{n}{q}}.$$
Therefore, iterating estimate \eqref{intermedia2} we deduce that the sequence $u_k$ is bounded in $W^{1,q}(B_R)$ and therefore its limit $u$ also belongs to $W^{1,q}(B_R)$. This conclude the proof.
\end{proof}
%
%

\end{document}